\documentclass[12pt]{amsart}
\usepackage{latexsym,amsmath, amscd, amssymb,amsthm, euscript, mathrsfs} 
\usepackage[all]{xy} 
\usepackage{hyperref}
\newtheorem{thm}[subsection]{Theorem}

\newtheorem{thm/def}[subsection]{Theorem/Definition}

\newtheorem{cor}[subsection]{Corollary}

\newtheorem{lem}[subsection]{Lemma}

\newtheorem{prop}[subsection]{Proposition}

\theoremstyle{definition}

\theoremstyle{definition}

\theoremstyle{definition}

\newtheorem{rem}[subsection]{Remark}

\numberwithin{equation}{subsection}

\newtheorem{pg}[subsection]{}

\newcommand{\Q}{\mathbb{Q}}

\newcommand{\R}{\mathbb R}

\newcommand{\Z}{\mathbb{Z}}

\newcommand{\Sp}{\text{\rm Spec}}

\newcommand{\dR}{\text{dR}}

\newcommand{\mls}{\mathscr}

\newcommand{\et}{\text{\rm\'et}}

\newcommand{\C}{\mathbb{C}}

\begin{document}

\title{A stronger derived Torelli theorem for K3 surfaces}
\author{Max Lieblich and Martin Olsson} 

\begin{abstract} 
  In an earlier paper the notion of a filtered derived
  equivalence was introduced, and it was shown that if two K3 surfaces admit such
  an equivalence then they are isomorphic. In this paper we study more refined
  aspects of filtered derived equivalences related to the action on the
  cohomological realizations of the Mukai motive. It is shown that if a filtered
  derived equivalence between K3 surfaces also preserves ample cones then one can
  find an isomorphism that induces the same map as the equivalence on the
  cohomological realizations.
\end{abstract}

\maketitle
\tableofcontents

\section{Introduction}

\begin{pg}   Let $k$ be an algebraically closed field of odd 
  characteristic and let $X$ and $Y$ be K3 surfaces over $k$.  Let $$
  \Phi :D(X)\rightarrow D(Y) $$ be an equivalence between their bounded
  triangulated categories of coherent sheaves given by a Fourier-Mukai kernel
  $P\in D(X\times Y)$, so $\Phi $ is the functor given by sending $M\in D(X)$
  to $$ R\text{pr}_{2*}(L\text{pr}_1^*M\otimes ^{\mathbb{L}}P).  $$ As
  discussed in \cite[2.9]{LO} the kernel $P$ also induces an isomorphism on
  rational Chow groups modulo numerical equivalence $$ \Phi
  _P^{A^*}:A^*(X)_{\text{num}, \Q}\rightarrow A^*(Y)_{\text{num}, \Q}.  $$

  We can consider how a given equivalence $\Phi$ interacts with the codimension
  filtration on $A^\ast$, or how it acts on the ample cone of $X$ inside $A^1(X)$.
  The underlying philosophy of this work is that tracking filtrations and ample
  cones (in ways we will make precise in Section \ref{S:stronglyfiltered}) gives a
  semi-linear algebraic gadget that behaves a lot like a Hodge structure. In
  Section \ref{S:stronglyfiltered} we will define a notion of \emph{strongly
    filtered} for an equivalence $\Phi$ that imposes conditions reminiscent of the
  classical Torelli theorem for K3 surfaces.

  With this in mind, the purpose of this paper is to prove the following result.  
\end{pg}

\begin{thm}\label{T:1.2} If $\Phi _P:D(X)\rightarrow D(Y)$ is a strongly
  filtered equivalence then there exists an isomorphism $\sigma :X\rightarrow
  Y$ such that the maps on the crystalline and \'etale realizations of the
  Mukai motive  induced by $\Phi _P$ and $\sigma $ agree.  
\end{thm}

For the definition of the realizations of the Mukai motive see
\cite[\S 2]{LO}. In 
\cite[Proof of 6.2]{LO} it is shown that any filtered equivalence can be modified to be
strongly filtered. As a consequence, we get a new proof of the following
result. 

\begin{thm}[{\cite[6.1]{LO}}]\label{T:1.1} If $\Phi _P^{A^*}$ preserves the
  codimension filtrations on $A^*(X)_{\text{\rm num}, \Q}$ and
  $A^*(Y)_{\text{\rm num}, \Q}$ then $X$ and $Y$ are isomorphic.  
\end{thm}

Whereas the original proof of Theorem \ref{T:1.1} relied heavily on liftings to
characteristic $0$ and Hodge theory, the proof presented here works primarily in
positive characteristic using  algebraic methods.

In Section \ref{S:section8} we present a proof of Theorem \ref{T:1.2} using certain
results about ``Kulikov models'' in positive characteristic (see Section
\ref{S:section4}). This argument implicitly uses Hodge theory which is an
ingredient in the proof of  Theorem \ref{T:2.2v1}. In Section \ref{S:section9} we discuss
a characteristic $0$ variant of Theorem \ref{T:1.2}, and finally the last
section \ref{S:section10} we explain how to bypass the use of the Hodge theory
ingredient of Theorem \ref{T:2.2v1}. This makes the argument entirely algebraic, except
for the Hodge theory aspects of the proof of the Tate conjecture. This also
gives a different algebraic perspective on the statement that any Fourier-Mukai
partner of a K3 surface is a moduli space of sheaves, essentially inverting
the methods of \cite{LO}. 

The bulk of this paper is devoted to proving Theorem \ref{T:1.2}. The basic idea is to
consider a certain moduli stack $\mls S_d$ classifying data $((X, \lambda ), Y,
P)$ consisting of a primitively polarized K3 surface $(X, \lambda )$ with
polarization of some degree $d$, a second K3 surface $Y$, and a complex $P\in
D(X\times Y)$ defining a strongly filtered Fourier-Mukai equivalence $\Phi
_P:D(X)\rightarrow D(Y)$. The precise definition is given in Section \ref{S:3},
where it is shown that $\mls S_d$ is an algebraic stack which is naturally a
$\mathbb{G}_m$-gerbe over a Deligne-Mumford stack $\overline {\mls S}_d$ \'etale
over the stack $\mls M_d$ classifying primitively polarized K3 surfaces of
degree $d$. The map $\overline {\mls S}_d\rightarrow \mls M_d$ is induced by the
map sending a collection $((X, \lambda ), Y, P)$ to $(X, \lambda )$. We then
study the locus of points in $\mls S_d$ where Theorem \ref{T:1.2} holds showing that it
is stable under both generization and specialization. From this it follows that
it suffices to consider the case when $X$ and $Y$ are supersingular where we can
use Ogus' crystalline Torelli theorem \cite[Theorem I]{Ogus2}.

\begin{rem} Our restriction to odd characteristic is because we appeal to the
  Tate conjecture for K3 surfaces, proven in odd characteristics by Charles,
  Maulik, and Pera \cite{Ch, Maulik, Pera}, which at present is not known in
  characteristic $2$.  
\end{rem}

\begin{pg} (Acknowledgements) Lieblich partially supported by NSF CAREER Grant
  DMS-1056129 and Olsson partially supported by NSF grant DMS-1303173 and a grant
  from The Simons Foundation. Olsson is grateful to F. Charles for inspiring
  conversations at the Simons Symposium ``Geometry Over Nonclosed Fields'' which
  led to results of this paper. We also thank E. Macr\`\i\, D. Maulik, and K. Pera for
  useful correspondence.
\end{pg}

\section{Strongly filtered equivalences}\label{S:stronglyfiltered}

\begin{pg} Let $X$ and $Y$ be K3 surfaces over an algebraically closed  field
  $k$ and let $P\in D(X\times Y)$ be an object defining an equivalence $$ \Phi
  _P:D(X)\rightarrow D(Y),  $$ and let $$ \Phi ^{A^*_{\text{num},
      \Q}}_P:A^*(X)_{\text{num}, \Q}\rightarrow A^*(Y)_{\text{num}, \Q} $$ denote the
  induced map on Chow groups modulo numerical equivalence and tensored with $\Q$.
  We say that $\Phi _P$ is \emph{filtered} (resp.\ \emph{strongly filtered},
  resp.\ \emph{Torelli}) if $\Phi ^{A^*_{\text{num}, \Q}}_P$ preserves the
  codimension filtration (resp.~is filtered, sends $(1, 0, 0)$ to $(1, 0, 0)$,
  and sends the ample cone of $X$ to plus or minus the ample cone of $Y$;
  resp.~is filtered, sends $(1, 0, 0)$ to $\pm (1, 0, 0)$, and sends the ample
  cone of $X$ to the ample cone of $Y$).  \end{pg} 

\begin{rem}
  Note that if $P$ is strongly filtered then either $P$ or $P[1]$ is Torelli. If
  $P$ is Torelli then either $P$ or $P[1]$ is strongly filtered.
\end{rem}
\begin{rem} Note that $A^1(X)$ is the orthogonal complement of
  $A^0(X)\oplus A^2(X)$ and similarly for $Y$.  This implies that if $\Phi _P$
  is filtered and sends $(1, 0, 0)$ to $\pm (1, 0, 0)$ then  $\Phi
  _P(A^1(X)_{\text{num}, \Q})\subset A^1(Y)_{\text{num}, \Q}$.  
\end{rem}

\begin{rem} It is shown in \cite[6.2]{LO} that if $\Phi _P:D(X)\rightarrow
  D(Y)$ is a filtered equivalence, then there exists a strongly filtered
  equivalence $\Phi :D(X)\rightarrow D(Y)$.  In fact it is shown there that
  $\Phi $ can be obtained from $\Phi _P$ by composing with a sequence of
  shifts, twists by line bundles, and spherical twists along $(-2)$-curves.
\end{rem}

\begin{pg} As noted in \cite[2.11]{LO} an equivalence $\Phi _P$ is filtered if
  and only if the induced map on Chow groups $$ \Phi ^{A^*_{\text{num},
      \Q}}_P:A^*(X)_{\text{num}, \Q}\rightarrow A^*(Y)_{\text{num}, \Q} $$ sends
  $A^2(X)_{\text{num}, \Q}$ to $A^2(X)_{\text{num}, \Q}$.  
\end{pg}

\begin{lem}\label{L:2.5b} Let $\ell $ be a prime invertible in $k$, let
  $\widetilde H(X, \Q_\ell )$ (resp.~$\widetilde H(Y, \Q_\ell )$) denote the
  $\Q_\ell $-realization of the Mukai motive of $X$ (resp.~$Y$) as defined in
  \cite[2.4]{LO}, and let $$ \Phi _P^{\et }:\widetilde H(X, \Q_\ell
  )\rightarrow \widetilde H(Y, \Q_\ell ) $$ denote the isomorphism defined by
  $P$.  Then $\Phi _P$ is filtered if and only if $\Phi _P^\et $ preserves the
  filtrations by degree on $\widetilde H(X, \Q_\ell )$ and $\widetilde H(Y,
  \Q_\ell )$.  
\end{lem} 

\begin{proof} By the same reasoning as in
  \cite[2.4]{LO} the map $\Phi _P^{\et }$ is filtered if and only if $$\Phi
  _P^\et (H^4(X, \Q_\ell )) = H^4(Y, \Q_\ell ).$$  Since the cycle class maps $$
  A^2(X)_{\text{num}, \Q}\otimes _{\Q}\Q_\ell \rightarrow H^4(X, \Q_\ell ), \ \
  A^2(Y)_{\text{num}, \Q}\otimes _{\Q}\Q_\ell \rightarrow H^4(Y, \Q_\ell ) $$
  are isomorphisms and the maps $\Phi _P$ and $\Phi _P^\et $ are compatible in
  the sense of \cite[2.10]{LO} it follows that if $\Phi _P$ is filtered then so
  is $\Phi _P^\et $.  Conversely if $\Phi _P^\et $ is filtered then since the
  cycle class maps $$ A^*(X)_{\text{num}, \Q}\rightarrow \widetilde H(X,
  \Q_\ell ), \ \ A^*(Y)_{\text{num}, \Q}\rightarrow \widetilde H(Y, \Q_\ell )
  $$ are injective it follows that $\Phi _P$ is also filtered.  
\end{proof}

\begin{rem} The same proof as in Lemma \ref{L:2.5b} gives variant results for
  crystalline cohomology and in characteristic $0$ de Rham cohomology.
\end{rem}

The condition that $\Phi _P$ takes the ample cone to plus or minus the ample
cone appears more subtle.  A useful observation in this regard is the
following. 

\begin{lem}\label{L:2.7}  Let $P\in D(X\times Y)$ be an object
  defining a filtered equivalence $\Phi _P:D(X)\rightarrow D(Y)$ such that
  $\Phi _P^{A^*_{\text{\rm num}}}$ sends $(1, 0, 0)$ to $(1, 0, 0)$.  Then
  $\Phi _P$ is strongly filtered if and only if for some ample invertible sheaf
  $L$ on $X$ the class $\Phi _P^{A^*_{\text{\rm num}}}(L)\in NS(Y)_\Q$ is plus
  or minus an ample class.  
\end{lem} 

\begin{proof} Following \cite[p.
  366]{Ogus2} define $$ V_X:= \{x\in NS(X)_\R|x^2>0, \ \text{and $\langle x,
    \delta \rangle  \neq 0$ for all $\delta \in NS(X)$ with $\delta ^2 = -2$}\},
  $$ and define $V_Y$ similarly.  Since $\Phi _P^{A^*_{\text{\rm num}}}$ is an
  isometry it induces an isomorphism $$ \sigma :V_X\rightarrow V_Y.  $$ By
  \cite[Proposition 1.10 and Remark 1.10.9]{Ogus2} the ample cone $C_X$ (resp.
  $C_Y$) of $X$ (resp. $Y$) is a connected component of $V_X$ (resp. $V_Y$) and
  therefore either $\sigma (C_X)\cap C_Y = \emptyset $ or $\sigma (C_X) = C_Y$,
  and similarly $\sigma (-C_X) \cap C_Y = \emptyset$ or $\sigma (-C_X) = C_Y$.
\end{proof}

\begin{prop}\label{P:2.8}  Let $X$ and $Y$ be K3-surfaces over a scheme $S$
  and let $P\in D(X\times _SY)$ be a relatively perfect complex.  Assume that
  $X/S$ is projective.  Then the set of points $s\in S$ for which the induced
  transformation on the derived category of the geometric fibers $$ \Phi
  _{P_{\bar s}}:D(X_{\bar s})\rightarrow D(Y_{\bar s}) $$ is a strongly
  filtered equivalence is an open subset of $S$.  
\end{prop} 

\begin{proof} By a
  standard reduction we may assume that $S$ is of finite type over $\Z$.

  First note that the condition that $\Phi _{P_{\bar s}}$ is an equivalence is an
  open condition.  Indeed as described  in \cite[discussion preceding 3.3]{LO}
  there exists a morphism of $S$-perfect complexes $\epsilon :P_1\rightarrow P_2$
  in $D(X\times _SY)$ such that $\Phi _{P_{\bar s}}$ is an equivalence if and
  only if $\epsilon _{\bar s}:P_{1, \bar s}\rightarrow P_{2, \bar s}$ is an
  isomorphism in $D(X_{\bar s}\times Y_{\bar s})$ (in loc. cit. we considered two
  maps of perfect complexes but one can just take the direct sum of these to get
  $\epsilon $).  Let $Q$ be the cone of $\epsilon $, and let $Z\subset X\times
  _SY$ be the support of the cohomology sheaves of $Q$.  Then the image of $Q$ in
  $S$ is closed and the complement of this image is the maximal open set over
  which the fiber transformations $\Phi _{P_{\bar s}}$ are equivalences.

  Replacing $S$ by an open set we may therefore assume that $\Phi _{P_{\bar s}}$
  is an equivalence in every fiber.

  Next we show that the condition that $\Phi _P$ is filtered is an open and
  closed condition.   For this we may assume we have a prime $\ell $ invertible
  in $S$.  Let $f_X:X\rightarrow S$ (resp. $f_Y:Y\rightarrow S$) be the structure
  morphism.  Define $\widetilde {\mls H}_{X/S}$ to be the lisse $\Q_\ell $-sheaf
  on $S$ given by $$ \widetilde {\mls H}_{X/S}:= (R^0f_{X*}\Q_\ell (-1))\oplus
  (R^2f_{X*}\Q_\ell )\oplus (R^4f_{X*}\Q_\ell )(1), $$ and define $\widetilde
  {\mls H}_{Y/S}$ similarly.  The kernel $P$ then induces a morphism of lisse
  sheaves $$ \Phi _{P/S}^{\et, \ell }:\widetilde {\mls H}_{X/S}\rightarrow
  \widetilde {\mls H}_{Y/S} $$ whose restriction to each geometric fiber is the
  map on the $\Q_\ell $-realization of the Mukai motive as in \cite[2.4]{LO}.  In
  particular, $\Phi _{P/S}^{\et, \ell }$ is an isomorphism.  By Lemma \ref{L:2.5b} for
  every geometric point $\bar s\rightarrow S$ the map $\Phi _{P_{\bar s}}$ is
  filtered if and only if the stalk $\Phi _{P/S, \bar s}^{\et, \ell }$ preserves
  the filtrations on $\widetilde {\mls H}_{X/S}$ and $\widetilde {\mls H}_{Y/S}$.
  In particular this is an open and closed condition on $S$. Shrinking on $S$ if
  necessary we may therefore further assume that $\Phi _{P_{\bar s}}$ is filtered
  for every geometric point $\bar s\rightarrow S$.

  It remains to show that in this case the set of points $s$ for which $\Phi _P$
  takes the ample cone $C_{X_{\bar s}}$ of $X_{\bar s}$ to $\pm C_{Y_{\bar s}}$
  is an open subset of $S$.  For this we can choose, by our assumption that $X/S$
  is projective, a relatively ample invertible sheaf $L$ on $X$.  Define $$ M:=
  \text{det}(R\text{pr}_{2*}(L\text{pr}_1^*(L)\otimes P)), $$ an invertible sheaf
  on $Y$.  Then by Lemma \ref{L:2.7} for a point $s\in S$ the transformation $\Phi
  _{P_{\bar s}}$ is strongly filtered if and only if the restriction of $M$ to
  the fiber $Y_{\bar s}$ is plus or minus the class of an ample divisor. By
  openness of the ample locus \cite[III, 4.7.1]{EGA} we get that being strongly
  filtered is an open condition.  
\end{proof}

\begin{prop}\label{P:2.10} 
Let $P\in D(X\times Y)$ be a complex such that the induced transformation 
$$
\Phi _P^{A^*_{\text{\rm num}, \Q}}:A^*(X)_{\text{\rm num}, \Q}\rightarrow A^*(X)_{\text{\rm num}, \Q}
$$
preserves the codimension filtration, takes $(1, 0, 0)$ to $(1, 0, 0)$, and takes the ample cone of $X$ to plus or minus the ample cone of $Y$ (so $P$ does not necessarily define an equivalence but otherwise behaves like a strongly filtered Fourier-Mukai equivalence).
  Suppose there exists an equivalence $\Phi _Q:D(X)\rightarrow D(Y)$ which is
  Torelli, and such that the induced map $NS(X)\rightarrow NS(Y)$ agrees with the
  map defined by $\pm \Phi _P$. Then $\Phi _P^{A^*_{\text{\rm num}, \Q}}$
  preserves the ample cones.
\end{prop}
\begin{proof}  
  Suppose that $\Phi _P$ takes the ample cone of $X$ to the negative of the ample
  cone of $Y$. Consider the auto-equivalence $\Phi := \Phi _Q^{-1}\circ \Phi
  _{P[1]}$ of $D(X)$. The induced automorphism
  $$
  \Phi ^{A^*_{\text{num}, \Q}}:A^*(X)_{\text{num} , \Q}\rightarrow
  A^*(X)_{\text{num}, \Q}
  $$
  then preserves the codimension filtration, Mukai pairing, and is the identity
  on $NS(X)_{\text{num} , \Q}$ and multiplication by $-1$ on $A^0(X)_{\text{num},
    \Q}$ and $A^2(X)_{\text{num}, \Q}$. By the compatibility of $\Phi $ with the
  Mukai pairing this implies that for any $H\in NS(X)$ we have
  $$
  -H^2 = \Phi \langle (0, H, 0), (0, H, 0)\rangle = \langle (0, H, 0), (0, H,
  0)\rangle = H^2,
  $$
  which is a contradiction. Thus $\Phi _{P[1]}$ must take $(0, 0, 1)$ to $(0, 0, 1)$
  which implies that $\Phi _{P[1]}$ takes $(1, 0,0)$ to $(1, 0, 0)$, a contradiction.
\end{proof}

\section{Moduli spaces of K3 surfaces}\label{S:3}

\begin{pg} For an integer $d$ invertible in $k$ let $\mls M_{d}$ denote the
  stack over $k$ whose fiber over a scheme $T$ is the groupoid of pairs $(X,
  \lambda )$ where $X/T$ is a proper smooth algebraic space all of whose
  geometric fibers are K3 surfaces and $\lambda :T\rightarrow
  \text{Pic}_{X/T}$ is a morphism to the relative Picard functor such that in
  every geometric fiber $\lambda $ is given by a primitive  ample line bundle
  $L_\lambda $ whose self-intersection is $2d$.  The following theorem
  summarizes the properties of the stack $\mls M_d$ that we will need.
\end{pg}

\begin{thm}\label{T:2.2} (i) $\mls M_d$ is a Deligne-Mumford stack, smooth over
  $k$ of relative dimension $19$.

  (ii) If $p\geq 3$ and $p^2\nmid d$ then the geometric fiber of $\mls M_d$ is
  irreducible.

  (iii) The locus $\mls M_{d, \infty }\subset \mls M_d$ classifying supersingular
  K3 surfaces is closed of dimension $\geq 9$.  
\end{thm} 

\begin{proof} A
  review of (i) and (iii) can be found in \cite[p.\ 1]{Ogus}.  Statement (ii) can
  be found in \cite[2.10 (3)]{Liedtke}.  
\end{proof}

\begin{rem} The stack $\mls M_d$ is defined over $\Z$, and it follows from (ii)
  that the geometric generic fiber of $\mls M_d$ is irreducible (this follows
  also from the Torelli theorem over $\C$ and the resulting description of
  $\mls M_{d, \C}$ as a period space). Furthermore over $\Z[1/d]$ the stack
  $\mls M_d$ is smooth. In what follows we denote this stack over $\Z[1/d]$ by
  $\mls M_{d, \Z[1/2]}$ and reserve the notation $\mls M_d$ for its reduction to
  $k$.
\end{rem}

\begin{rem} Note that in the definition of $\mls M_d$ we consider ample
  invertible sheaves, and don't allow contractions in the corresponding morphism
  to projective space.
\end{rem}

\begin{pg} Let $\mls S_d$ denote the fibered category  over $k$ whose fiber
  over a scheme $S$ is the groupoid of collections of data
  \begin{equation}\label{E:theobject} ((X, \lambda ), Y, P), \end{equation}
  where $(X, \lambda )\in \mls M_{d}(S)$ is a polarized K3 surface, $Y/S$ is
  a second K3 surface over $S$, and $P\in D(X\times _SY)$ is an $S$-perfect
  complex such that for every geometric point $\bar s\rightarrow S$ the induced
  functor $$ \Phi ^{P_{\bar s}}:D(X_{\bar s})\rightarrow D(Y_{\bar s}) $$ is
  strongly filtered.  
\end{pg}

\begin{thm}\label{T:4.5} The fibered category $\mls S_d$ is an algebraic stack
  locally of finite type over $k$.  
\end{thm} 

\begin{proof}   By fppf descent
  for algebraic spaces we have descent for both polarized and unpolarized K3
  surfaces.  

  To verify descent for the kernels $P$, consider an object \eqref{E:theobject}
  over a scheme $S$.  Let $P^\vee $ denote $\mls RHom (P, \mls O_X)$.  Since $P$
  is a perfect complex we have $\mls RHom (P, P)\simeq P^\vee \otimes P$.  By
  \cite[2.1.10]{L} it suffices to show that for all geometric points $\bar
  s\rightarrow S$ we have $H^i(X_{\bar s}\times Y_{\bar s}, P^\vee _{\bar
    s}\otimes P_{\bar s}) = 0$ for $i<0$.    This follows from the following
  result (we discuss Hochschild cohomology further in Section \ref{S:section4b} below):

  \begin{lem}[{\cite[5.6]{Toda}, \cite[5.1.8]{Grigg}}]\label{L:2.5} Let $X$ and
    $Y$ be K3 surfaces over an algebraically closed field $k$, and let $P\in
    D(X\times Y)$ be a complex defining a Fourier-Mukai equivalence $\Phi
    _P:D(X)\rightarrow D(Y)$.    Denote by $HH^*(X)$ the Hochschild cohomology of
    $X$ defined as $$ \text{\rm RHom}_{X\times X}(\Delta _*\mls O_X, \Delta
    _*\mls O_X).  $$

    (i) There is a canonical isomorphism $\text{\rm Ext}^*_{X\times Y}(P, P)\simeq
    HH^*(X)$.

    (ii) $\text{\rm Ext}^i_{X\times Y}(P, P) = 0$ for $i<0$ and $i=1$.

    (iii)  The  natural map $k\rightarrow \text{\rm Ext}^0_{X\times Y}(P, P)$ is an
    isomorphism.  
  \end{lem} 

  \begin{proof} Statement (i) is \cite[5.6]{Toda}.
    Statements (ii) and (iii) follow immediately from this, since $HH^1(X) = 0$ for
    a K3 surface.  
  \end{proof}

  Next we show that for an object \eqref{E:theobject} the polarization $\lambda $
  on $X$ induces a polarization $\lambda _Y$ on $Y$.  To define $\lambda _Y$ we
  may work \'etale locally on $S$ so may assume there exists an ample invertible
  sheaf $L$ on $X$ defining $\lambda $.  The complex $$ \Phi _P(L):=
  R\text{pr}_{2*}(\text{pr}_1^*L\otimes ^{\mathbb{L}}P) $$ is $S$-perfect, and
  therefore a perfect complex on $Y$.  Let $M$ denote the determinant of $\Phi
  _P(L)$, so $M$ is an invertible sheaf on $Y$.  By our assumption that $\Phi
  ^{P_s}$ is strongly filtered for all $s\in S$, the restriction of $M$ to any
  fiber is either ample or antiample.  It follows that either $M$ or $M^\vee $ is
  a relatively ample invertible sheaf and we define $\lambda _Y$ to be the
  resulting polarization on $Y$.  Note that this does not depend on the choice of
  line bundle $L$ representing $\lambda $ and therefore by descent $\lambda _Y$
  is defined even when no such $L$ exists.

  The degree of $\lambda _Y$ is equal to $d$.  Indeed if $s\in S$ is a point then
  since $\Phi ^{P_s}$ is strongly filtered the induced map $NS(X_{\bar
    s})\rightarrow NS(Y_{\bar s})$ is compatible with the intersection pairings and
  therefore $\lambda _Y^2 = \lambda ^2 = 2d$.

  From this we deduce that $\mls S_d$ is algebraic as follows.   We have a
  morphism \begin{equation}\label{E:2.4.1} \mls S_d\rightarrow \mls M_d\times
    \mls M_d, \ \ ((X, \lambda ), Y, P)\mapsto ((X, \lambda ), (Y, \lambda _Y)),
  \end{equation} and $\mls M_d\times \mls M_d$ is an algebraic stack.  Let $\mls
  X$ (resp. $\mls Y$) denote the pullback to $\mls M_d\times \mls M_d$ of the
  universal family over the first factor  (resp. second factor).  Sending a
  triple $((X, \lambda ), Y, P)$ to $P$ then realizes $\mls S_d$ as an open
  substack of the stack over $\mls M_d\times \mls M_d$ of simple universally
  gluable complexes on $\mls X\times _{\mls M_d\times \mls M_d}\mls Y$ (see for
  example \cite[\S 5]{LO}).  
\end{proof}

\begin{pg} Observe that for any object $((X, \lambda ), Y, P)\in \mls S_d$ over
  a scheme $S$ there is an inclusion $$ \mathbb{G}_m\hookrightarrow \underline
  {\text{Aut}}_{\mls S_d}((X, \lambda ), Y, P) $$ giving by scalar
  multiplication by $P$.   We can therefore form the rigidification of $\mls
  S_d$ with respect to $\mathbb{G}_m$ (see for example \cite[\S 5]{ACV}) to get
  a morphism $$ g:\mls S_d\rightarrow \overline {\mls S}_d $$ realizing $\mls
  S_d$ as a $\mathbb{G}_m$-gerbe over another algebraic stack $\overline {\mls
    S}_d$.  By the universal property of rigidification the map $\mls
  S_d\rightarrow \mls M_d$ sending $((X, \lambda ), Y, P)$ to $(X, \lambda )$
  induces a morphism \begin{equation}\label{E:etalemap} \pi :\overline {\mls
      S}_d\rightarrow \mls M_d.  \end{equation} 
\end{pg}

\begin{thm}\label{T:4.8} The stack $\overline {\mls S}_d$ is Deligne-Mumford
  and the map \eqref{E:etalemap} is \'etale.  
\end{thm} 

\begin{proof} Consider
  the map \ref{E:2.4.1}.  By the universal property of rigidification this
  induces a morphism $$ q:\overline {\mls S}_d\rightarrow \mls M_d\times \mls
  M_d.  $$ Since $\mls M_d\times \mls M_d$ is Deligne-Mumford, to prove that
  $\overline {\mls S}_d$ is a Deligne-Mumford stack it suffices to show that
  $q$ is representable.  This follows from Lemma \ref{L:2.5} (iii) which implies that
  for any object $((X, \lambda ), Y, P)$ over a scheme $S$ the automorphisms of
  this object which map under $q$ to the identity are given by scalar
  multiplication on $P$ by elements of $\mls O_S^*$.

  It remains to show that the map \eqref{E:etalemap} is \'etale, and for this it
  suffices to show that it is formally \'etale. 

  Let $A\rightarrow A_0$ be a surjective map of artinian local rings with kernel
  $I$ annhilated by the maximal ideal of $A$, and let $k$ denote the residue
  field of $A_0$ so $I$ can be viewed as a $k$-vector space.  Let $((X_0, \lambda
  _0), Y_0, P_0)\in \mls S_d(A_0)$ be an object and let $(X, \lambda )\in \mls
  M_d(A)$ be a lifting of $(X_0, \lambda _0)$ so we have a commutative diagram of
  solid arrows $$ \xymatrix{ \Sp (A_0)\ar@{^{(}->}[dd]^-i\ar[r]^-{x_0}& \mls
    S_d\ar[d]\\ & \overline {\mls S}_d\ar[d]\\ \Sp (A)\ar@{-->}[ru]^-{\bar
      x}\ar@{-->}[ruu]^-x\ar[r]^-{y}& \mls M_d.} $$ Since $\mls S_d$ is a
  $\mathbb{G}_m$-gerbe over $\overline {\mls S}_d$, the obstruction to lifting a
  map $\bar x$ as indicated to a morphism $x$ is given by a class in $H^2(\Sp
  (A), \widetilde I) = 0$, and therefore any such map $\bar x$ can be lifted to a
  map $x$.  Furthermore, the set of isomorphism classes of such liftings $x$ of
  $\bar x$ is given by $H^1(\Sp (A), \widetilde I) = 0$ so in fact the lifting
  $x$ is unique up to isomorphism.  The isomorphism is not unique but determined
  up to the action of $$ \text{Ker}(A^*\rightarrow A_0^*) \simeq I.  $$ From this
  it follows that it suffices to show the following: 

  \begin{enumerate} 
  \item [(i)] The lifting $(X, \lambda )$ of $(X_0, \lambda _0)$ can be extended
    to a lifting $((X, \lambda ), Y, P)$ of $((X_0, \lambda _0), Y_0, P_0)$.
  \item [(ii)] This extension $((X, \lambda ), Y, P)$ of $(X, \lambda )$ is
    unique up to isomorphism.  \item [(iii)] The automorphisms of the triple
    $((X, \lambda ), Y, P)$ which are the identity on $(X, \lambda )$ and
    reduce to the identity over $A_0$ are all given by scalar multiplication
    on $P$ by elements of $1+I\subset A^*$.  
  \end{enumerate} 

  Statement (i) is
  shown in \cite[6.3]{LO}.

  Next we prove the uniqueness statements in (ii) and (iii).  Following the
  notation of \cite[Discussion preceding 5.2]{LO}, let $s\mls D_{X/A}$ denote the
  stack of simple, universally gluable, relatively perfect complexes on $X$, and
  let $sD_{X/A}$ denote its rigidifcation with respect to the
  $\mathbb{G}_m$-action given by scalar multiplication.    The complex $P_0$ on
  $X_0\times _{A_0}Y_0$ defines a morphism $$ Y_0\rightarrow sD_{X/A}\otimes
  _AA_0 $$ which by \cite[5.2 (ii)]{LO} is an open imbedding.  Any extension of
  $(X, \lambda )$ to a lifting $((X, \lambda ), Y, P)$  defines an open imbedding
  $Y\hookrightarrow sD_{X/A}$.  This implies that $Y$, viewed as a deformation of
  $Y_0$ for which there exists a lifting $P$ of $P_0$ to $X\times _AY$, is unique
  up to unique isomorphism.

  Let $Y$ denote the unique lifting of $Y_0$ to an open subspace of $sD_{X/A}$.
  By \cite[3.1.1 (2)]{L} the set of isomorphism classes of liftings of $P_0$ to
  $X\times _AY$ is a torsor under $$ \text{Ext}^1_{X_k\times Y_k}(P_k, P_k)\otimes I, $$
  which is $0$ by Lemma \ref{L:2.5} (ii).  From this it follows that $P$ is unique up
  to isomorphism, and also by Lemma \ref{L:2.5} (iii) we get the statement that the
  only infinitesimal automorphisms of the triple $((X, \lambda ), Y, P)$ are
  given by scalar multiplication by elements of $1+I$.  
\end{proof}

\begin{pg} There is an automorphism $$ \sigma :\mls S_d\rightarrow \mls S_d $$
  satisfying $\sigma ^2 = \text{id}$.  This automorphism is defined by sending
  a triple $((X, \lambda ), Y, P)$ to $((Y, \lambda _Y), X, P^\vee  [2])$.
  This automorphism induces an involution $\bar \sigma :\overline {\mls
    S}_d\rightarrow \overline {\mls S}_d$ over the involution $\gamma :\mls
  M_d\times \mls M_d\rightarrow \mls M_d\times \mls M_d$ switching the factors.
\end{pg}

\begin{rem}\label{R:4.10} In fact the stack $\mls S_d$ is defined over
  $\Z[1/d]$ and Theorems \ref{T:4.5} and \ref{T:4.8} also hold over $\Z[1/d]$.
  In what follows we write $\mls S_{d, \Z[1/d]}$ for this stack over $\Z[1/d]$.
\end{rem}

\section{Deformations of autoequivalences}\label{S:section4b}

In this section, we describe the obstructions to deforming
Fourier-Mukai equivalences. The requisite technical machinery for this is worked
out in \cite{HMS} and \cite{HT}. The results of this section will play a crucial
role in Section \ref{sec:supers-reduct}. 

Throughout this section let $k$ be a perfect field of positive characteristic
$p$ and ring of Witt vectors $W$. For an integer $n$ let $R_n$ denote the ring
$k[t]/(t^{n+1})$, and let $R$ denote the ring $k[[t]]$.

\begin{pg} Let $X_{n+1}/R_{n+1}$ be a smooth proper scheme over $R_{n+1}$ with
  reduction $X_n$ to $R_n$. We then have the associated \emph{relative
    Kodaira-Spencer class}, defined in
  \cite[p. 486]{HMS}, which is the morphism in $D(X_n)$
  $$
  \kappa _{X_n/X_{n+1}}:\Omega ^1_{X_n/R_n}\rightarrow \mls O_{X_n}[1]
  $$
  defined as the morphism corresponding to the short exact sequence
  $$
  \xymatrix{ 0\ar[r]& \mls O_{X_n}\ar[r]^-{\cdot dt}& \Omega
    ^1_{X_{n+1}/k}|_{X_n}\ar[r]& \Omega ^1_{X_n/R_n}\ar[r]& 0.}
  $$
\end{pg}

\begin{pg} We also have the \emph{relative universal Atiyah class} which is a
  morphism
  $$
  \alpha _n:\mls O_{\Delta _n}\rightarrow i_{n*}\Omega ^1_{X_n/R_n}[1]
  $$
  in $D(X_n\times _{R_n}X_n)$, where $i_n:X_n\rightarrow X_n\times
  _{R_n}X_n$ is the diagonal morphism and $\mls O_{\Delta _n}$ denotes $i_{n*}\mls
  O_{X_n}$.

  This map $\alpha _n$ is given by the class of the short exact sequence
  $$
  0\rightarrow I/I^2\rightarrow \mls O_{X_n\times _{R_n}X_n}/I^2\rightarrow \mls
  O_{\Delta _n}\rightarrow 0,
  $$
  where $I\subset \mls O_{X_n\times _{R_n}X_n}$ is the ideal of the diagonal.
  Note that to get the morphism $\alpha _n$ we need to make a choice of
  isomorphism $I/I^2\simeq \Omega ^1_{X_n/R_n}$, which implies that the relative
  universal Atiyah class is not invariant under the map switching the factors, but
  rather changes by $-1$.
\end{pg}

\begin{pg} Define the \emph{relative Hochschild cohomology} of $X_n/R_n$ by
  $$
  HH^*(X_n/R_n):= \text{Ext}^*_{X_n\times _{R_n}X_n}(\mls O_{\Delta _n}, \mls
  O_{\Delta _n}).
  $$
  The composition
  $$
  \xymatrix{ \mls O_{\Delta _n}\ar[r]^-{\alpha _n}& i_{n*}\Omega
    ^1_{X_n/R_n}[1]\ar[rr]^-{i_{n*}\kappa _{X_n/X_{n+1}}}& &\mls O_{\Delta _n}[2]}
  $$
  is a class
  $$
  \nu _{X_n/X_{n+1}}\in HH^2(X_n/R_n).
  $$
\end{pg}

\begin{pg}\label{P:4.4b} Suppose now that $Y_{n}/R_n$ is a second smooth proper
  scheme with a smooth lifting $Y_{n+1}/R_{n+1}$ and that $E_n\in D(X_n\times
  _{R_n}Y_n)$ is a $R_n$-perfect complex.

  Consider the class
  $$
  \nu := \nu _{X_n\times _{R_n}Y_n/X_{n+1}\times _{R_{n+1}}Y_{n+1}}:\mls
  O_{\Delta _{n, X_n\times _{R_n}Y_n}}\rightarrow \mls O_{\Delta _{n, X_n\times
      _{R_n}Y_n}}[2].
  $$
  Viewing this is a morphism of Fourier-Mukai kernels
  $$
  D(X_n\times _{R_n}Y_n)\rightarrow D(X_n\times _{R_n}Y_n)
  $$
  and applying it to $E_n$ we get a class
  $$
  \omega (E_n)\in \text{Ext}^2_{X_n\times _{R_n}Y_n}(E_n, E_n).
  $$
  In the case when
  $$
  \text{Ext}^1_{X_0\times Y_0}(E_0, E_0) = 0,
  $$
  which will hold in the cases of interest in this paper, we know by \cite[Lemma
  3.2]{HMS} that the class $\omega (E_n)$ is $0$ if and only if $E_n$ lifts to a
  perfect complex on $X_{n+1}\times _{R_{n+1}}Y_{n+1}$.
\end{pg}

\begin{pg} To analyze the class $\omega (E_n)$ it is useful to translate it into
  a statement about classes in $HH^2(Y_n/R_n)$. This is done using Toda's argument
  \cite[Proof of 5.6]{Toda}. Let
  $$
  E_n\circ :D(X_n\times _{R_n}X_n)\rightarrow D(X_n\times _{R_n}Y_n)
  $$
  denote the map sending an object $K\in D(X_n\times _{R_n}X_n)$ to the complex
  representing the Fourier-Mukai transform $\Phi _{E_n}\circ \Phi _K$. Explicitly
  it is given by the complex
  $$
  p_{13*}(p_{12}^*K\otimes p_{23}^*E_n),
  $$
  where $p_{ij}$ denote the various projections from $X_n\times _{R_n}X_n\times
  _{R_n}Y_n$. As in loc.~cit.~the diagram
  $$
  \xymatrix{ D(X_n)\ar[d]^-{i_{n*}}\ar[rd]^-{p_1^*(-)\otimes E_n}& \\
    D(X_n\times _{R_n}X_n)\ar[r]^-{E_n\circ }& D(X_n\times _{R_n}Y_n)}
  $$
  commutes.

  In particular we get a morphism
  $$
  \eta _{X}^*:HH^*(X_n/R_n)\rightarrow \text{Ext}^*_{X_n\times _{R_n}Y_n}(E_n,
  E_n).
  $$

  Now assume that both $X_n$ and $Y_n$ have relative dimension $d$ over $R_n$
  and that the relative canonical sheaves of $X_n$ and $Y_n$ over $R_n$ are trivial. Let $E_n^\vee $ denote
  $\mls RHom (E_n, \mls O_{X_n\times _{R_n}Y_n})$ viewed as an object of $D(Y_n\times _{R_n}X_n)$.  In this
  case the functor
  $$
  \Phi _{E_n^\vee [d]}:D(Y_n)\rightarrow D(X_n)
  $$
  is both a right and left adjoint of $\Phi _{E_n}$ \cite[4.5]{Bridgeland}.  By the same argument, the functor
  $$
  \circ E_n^\vee [d]:D(X_n\times _{R_n}Y_n)\rightarrow D(Y_n\times _{R_n}Y_n),
  $$
  defined in the same manner as $E_n\circ $ has left and right adjoint given by
  $$
  \circ E_n:D(Y_n\times _{R_n}Y_n)\rightarrow D(X_n\times _{R_n}Y_n).
  $$
  Composing with the adjunction maps
  \begin{equation}\label{E:adjunctionmaps}
    \alpha :\text{id}\rightarrow \circ E_n\circ E_n^\vee [d], \ \ \beta :\circ E_n\circ E_n^\vee
    [d]\rightarrow \text{id}
  \end{equation}
  applied to the diagonal $\mls O_{\Delta _{Y_n}}$
  we get a morphism
  $$
  \eta _{Y*}:\text{Ext}^*_{X_n\times _{R_n}Y_n}(E_n, E_n)\rightarrow
  HH^*(Y_n/R_n).
  $$
  We denote the composition
  $$
  \eta _{Y*}\eta _X^*:HH^*(X_n/R_n)\rightarrow HH^*(Y_n/R_n)
  $$
  by $\Phi _{E_n}^{HH^*}$. In the case when $E_n$ defines a Fourier-Mukai
  equivalence this agrees with the standard definition (see for example
  \cite{Toda}).
\end{pg}

\begin{pg}
  Evaluating the adjunction maps \eqref{E:adjunctionmaps} on $\mls O_{\Delta _{Y_n}}$ we get a morphism
  \begin{equation}\label{E:4.6.1}
    \xymatrix{ \mls O_{\Delta _{Y_n}}\ar[r]^-{\alpha }&\mls O_{\Delta
        _{Y_n}}\circ E_n\circ E_n^\vee [d]\ar[r]^-{\beta }& \mls O_{\Delta _{Y_n}}.}
  \end{equation}
  We say that $E_n$ is \emph{admissible} if this composition is the identity map.

  If $E_n$ is a Fourier-Mukai equivalence then it is clear that $E_n$ is
  admissible. Another example is if there exists a lifting $(\mls X, \mls Y, \mls
  E)$ of $(X_n, Y_n, E_n)$ to $R$, where $\mls X$ and $\mls Y$ are smooth proper
  $R$-schemes with trivial relative canonical bundles and $\mls E$ is a
  $R$-perfect complex on $\mls X\times _R\mls Y$, such that the restriction $\mls
  E$ to the generic fiber defines a Fourier-Mukai equivalence. Indeed in this case
  the map \eqref{E:4.6.1} is the reduction of the corresponding map $\mls
  O_{\Delta _{\mls Y}}\rightarrow \mls O_{\Delta _{\mls Y}}$ defined over $R$,
  which in turn is determined by its restriction to the generic fiber.
\end{pg}
\begin{pg} Consider Hochschild homology
  $$
  HH_i(X_n/R_n):= H^{-i}(X_n, Li_n^*\mls O_{\Delta _n}).
  $$ 
  By the argument of \cite[\S 5]{Caldararu1} we also get an action
  $$
  \Phi _{E_n}^{HH_*}:HH_*(X_n/R_n)\rightarrow HH_*(Y_n/R_n).
  $$
  Hochschild homology is a module over Hochschild cohomology, and an
  exercise (that we do not write out here) shows that the following diagram
  $$
  \xymatrixcolsep{5pc}\xymatrix{ HH^*(X_n/R_n)\times HH_*(X_n/R_n)\ar[r]^-{\Phi _{E_n}^{HH^*}\times
      \Phi _{E_n}^{HH_*}}\ar[d]^-{\text{mult}}&HH^*(Y_n/R_n)\times
    HH_*(Y_n/R_n)\ar[d]^-{\text{mult}}\\ HH_*(X_n/R_n)\ar[r]^-{\Phi _{E_n}^{HH_*}}&
    HH_*(Y_n/R_n)}
  $$
  commutes.
\end{pg}

\begin{pg} Using this we can describe the obstruction $\omega (E_n)$ in a
  different way, assuming that $E_n$ is admissible. First note that viewing the relative Atiyah class of $X_n\times
  _{R_n}Y_n$ as a morphism of Fourier-Mukai kernels we get the Atiyah class of
  $E_n$ which is a morphism
  $$
  A(E_n):E_n\rightarrow E_n\otimes \Omega ^1_{X_n\times _{R_n}Y_n/R_n}[1]
  $$
  in $D(X_n\times _{R_n}Y_n)$. There is a natural decomposition
  $$
  \Omega ^1_{X_n\times _{R_n}Y_n/R_n}\simeq p_1^*\Omega ^1_{X_n/R_n}\oplus
  p_2^*\Omega ^1_{Y_n/R_n},
  $$
  so we can write $A(E_n)$ as a sum of two maps
  $$
  A(E_n)_X:E_n\rightarrow E_n\otimes p_1^*\Omega ^1_{X_n/R_n}[1], \ \
  A(E_n)_Y:E_n\rightarrow E_n\otimes p_2^*\Omega ^1_{X_n/R_n}[1].
  $$
  Similarly the Kodaira-Spencer class of $X_n\times _{R_n}Y_n$ can be written as
  the sum of the two pullbacks
  $$
  p_1^*\kappa _{X_n/X_{n+1}}:p_1^*\Omega ^1_{X_n/R_n}\rightarrow p_1^*\mls
  O_{X_n}[1], \ \ p_2^*\kappa _{Y_n/Y_{n+1}}:p_2^*\Omega ^1_{Y_n/R_n}\rightarrow
  p_2^*\mls O_{Y_n}[1].
  $$
  It follows that the obstruction $\omega (E_n)$ can be written as a sum
  $$
  \omega (E_n) = (p_1^*\kappa _{X_n/X_{n+1}}\circ A(E_n)_X)+ (p_2^*\kappa
  (Y_n/Y_{n+1})\circ A(E_n)_Y).
  $$
  Now by construction we have
  $$
  \eta _{X_n}^*(\nu _{X_n/X_{n+1}}) = p_1^*\kappa _{X_n/X_{n+1}}\circ A(E_n)_X,
  $$
  and
  $$
  \eta _{Y_n*}(p_2^*\kappa (Y_n/Y_{n+1})\circ A(E_n)_Y) = -\nu _{Y_n/Y_{n+1}},
  $$
  the sign coming from the asymmetry in the definition of the relative Atiyah
  class (it is in the verification of this second formula that we use the assumption that $E_n$ is admissible). Summarizing we find the formula
  \begin{equation}\label{E:keyformula} \eta _{Y_n*}(\omega (E_n)) = \Phi
    _{E_n}^{HH^*}(\nu _{X_n/X_{n+1}})-\nu _{Y_n/Y_{n+1}}.
  \end{equation}

  In the case when $\Phi _{E_n}$ is an equivalence the maps $\eta _{Y_n*}$ and
  $\eta _{X_n}^*$ are isomorphisms, so the obstruction $\omega (E_n)$ vanishes if
  and only if we have
  $$
  \Phi _{E_n}^{HH^*}(\nu _{X_n/X_{n+1}})-\nu _{Y_n/Y_{n+1}} = 0.
  $$
\end{pg}

\begin{rem} By \cite[Remark 2.3 (iii)]{HMS}, the functor $\Phi _{E_n}$ is an
  equivalence if and only if $\Phi _{E_0}:D(X_0)\rightarrow D(Y_0)$ is an
  equivalence.
\end{rem}

\begin{cor} Suppose $F_n\in D(X_n\times _{R_n}Y_n)$ defines a Fourier-Mukai
  equivalence, and that $E_n\in D(X_n\times _{R_n}Y_n)$ is another admissible $R_n$-perfect
  complex such that $\Phi _{F_n}^{HH^*} = \Phi _{E_n}^{HH^*}$. If $E_n$ lifts to a
  $R_{n+1}$-perfect complex $E_{n+1}\in D(X_{n+1}\times _{R_{n+1}}Y_{n+1})$ then
  so does $F_n$.
\end{cor}
\begin{proof} Indeed the condition that $\Phi _{F_n}^{HH^*} = \Phi
  _{E_n}^{HH^*}$ ensures that
  $$
  \eta _{Y_n*}(\omega (E_n) ) = \eta _{Y_n*}(\omega (F_n)),
  $$
  and since $\omega (E_n) = 0$ we conclude that $\omega (F_n) = 0$.
\end{proof}




\begin{pg} The next step is to understand the relationship between $\Phi
  _{E_n}^{HH^*}$ and the action of $\Phi _{E_n}$ on the cohomological realizations
  of the Mukai motive.

  Assuming that the characteristic $p$ is bigger than the dimension of $X_0$
  (which in our case will be a K3 surface so we just need $p>2$) we can
  exponentiate the relative Atiyah class to get a map
  $$
  \exp (\alpha _n):\mls O_{\Delta _n}\rightarrow \oplus _ii_{n*}\Omega
  ^i_{X_n/R_n}
  $$
  which by adjunction defines a morphism
  \begin{equation}\label{E:HKRiso} Li_n^*\mls O_{\Delta _n}\rightarrow \oplus
    _i\Omega ^i_{X_n/R_n}
  \end{equation} in $D(X_n)$. By \cite[Theorem 0.7]{AC}, which also holds
  in positive characteristic subject to the bounds on dimension, this map is an
  isomorphism. We therefore get an isomorphism
  $$
  I^{HKR}:HH^*(X_n/R_n)\rightarrow HT^*(X_n/R_n),
  $$
  where we write
  $$
  HT^*(X_n/R_n) := \oplus _{p+q=*}H^p(X_n, \bigwedge ^qT_{X_n/R_n}).
  $$
  We write
  $$
  I^K_{X_n}:HH^*(X_n/R_n)\rightarrow HT^*(X_n/R_n)
  $$
  for the composition of $I^{HKR}$ with multiplication by the inverse square root of
  the Todd class of $X_n/R_n$, as in \cite[1.7]{Caldararu2}.

  The isomorphism \eqref{E:HKRiso} also defines an isomorphism
  $$
  I_{HKR}:HH_*(X_n/R_n)\rightarrow H\Omega _*(X_n/R_n),
  $$
  where
  $$
  H\Omega _*(X_n/R_n):= \oplus _{q-p=*}H^p(X_n, \Omega ^q_{X_n/R_n}).
  $$
  We write
  $$
  I_{K}^{X_n}:HH_*(X_n/R_n)\rightarrow H\Omega _*(X_n/R_n)
  $$
  for the composition of $I_{HKR}$ with multiplication by the square root of the
  Todd class of $X_n/R_n$.

  We will consider the following condition $(\star )$ on a $R_n$-perfect complex
  $E_n\in D(X_n\times _{R_n}Y_n)$:
\end{pg}

\begin{enumerate}
\item [$(\star )$] The diagram
  $$
  \xymatrix{ HH_*(X_n/R_n)\ar[r]^-{\Phi _{E_n}^{HH_*}}\ar[d]^-{I_K^{X_n}}&
    HH_*(Y_n/R_n)\ar[d]^-{I_K^{Y_n}}\\ H\Omega _*(X_n/R_n)\ar[r]^-{\Phi
      _{E_n}^{H\Omega _*}}& H\Omega _*(Y_n/R_n)}
  $$
  commutes.
\end{enumerate}

\begin{rem} We expect this condition to hold in general. Over a field of
  characteristic $0$ this is shown in \cite[1.2]{MS}. We expect that a careful
  analysis of denominators occurring of their proof will verify $(\star )$ quite
  generally with some conditions on the characteristic relative to the dimension
  of the schemes. However, we will not discuss this further in this paper.
\end{rem}

\begin{pg} There are two cases we will consider in this paper were $(\star )$ is
  known to hold:
  \begin{enumerate}
  \item [(i)] If $E_n = \mls O_{\Gamma _n}$ is the structure sheaf of the graph
    of an isomorphism $\gamma _n:X_n\rightarrow Y_n$. In this case the induced maps
    on Hochschild cohomology and $H\Omega _*$ are simply given by pushforward
    $\gamma _{n*}$ and condition $(\star )$ immediately holds.
  \item [(ii)] Suppose $B\rightarrow R_n$ is a morphism from an integral domain
    $B$ which is flat over $W$ and that there exists a lifting $(\mls X, \mls Y,
    \mls E)$ of $(X_n, Y_n, E_n)$ to $B$, where $\mls X$ and $\mls Y$ are proper and
    smooth over $B$ and $\mls E\in D(\mls X\times _B\mls Y)$ is a $B$-perfect
    complex pulling back to $E_n$. Suppose further that the groups $HH^*(\mls X/B)$
    and $HH^*(\mls Y/B)$ are flat over $B$ and their formation commutes with base
    change (this holds for example if $\mls X$ and $\mls Y$ are K3 surfaces). Then
    $(\star )$ holds. Indeed it suffices to verify the commutativity of the
    corresponding diagram over $B$, and this in turn can be verified after passing
    to the field of fractions of $B$. In this case the result holds by
    \cite[1.2]{MS}.
  \end{enumerate}
\end{pg}

\begin{lem}\label{L:4.13} Let $E_n, F_n\in D(X_n\times _{R_n}Y_n)$ be two
  $R_n$-perfect complexes satisfying condition $(\star )$. Suppose further that
  the maps $\Phi _{E_0}^{\text{\rm crys} }$ and $\Phi _{F_0}^\text{\rm crys} $ on
  the crystalline realizations $\widetilde H(X_0/W)\rightarrow \widetilde
  H(Y_0/W)$ of the Mukai motive are equal. Then the maps $\Phi _{E_n}^{HH_*}$ and
  $\Phi _{F_n}^{HH_*}$ are also equal. Furthermore if the maps on the crystalline
  realizations are isomorphisms then $\Phi _{E_n}^{HH_*}$ and $\Phi _{F_n}^{HH_*}$
  are also isomorphisms.
\end{lem}
\begin{proof} Since $H\Omega _*(X_n/R_n)$ (resp.~$H\Omega _*(Y_n/R_n)$) is
  obtained from the de Rham realization $\widetilde H_{\text{dR}}(X_n/R_n)$
  (resp.~$\widetilde H_{\text{dR}}(X_n/R_n)$) of the Mukai motive of $X_n/R_n$
  (resp.~$Y_n/R_n$) by passing to the associated graded, it suffices to show that
  the two maps
  $$
  \Phi _{E_n}^{\text{dR}}, \Phi _{F_n}^{\text{dR}}: \widetilde
  H_{\text{dR}}(X_n/R_n)\rightarrow \widetilde H_{\text{dR}}(Y_n/R_n)
  $$
  are equal, and isomorphisms when the crystalline realizations are
  isomorphisms. By the comparison between crystalline and de Rham cohomology it
  suffices in turn to show that the two maps on the crystalline realizations
  $$
  \Phi _{E_n}^{\text{crys}}, \Phi _{F_n}^{\text{crys}}: \widetilde
  H_{\text{crys}}(X_n/W[t]/(t^{n+1}))\rightarrow \widetilde
  H_{\text{crys}}(Y_n/W[t]/(t^{n+1}))
  $$
  are equal. Via the Berthelot-Ogus isomorphism \cite[2.2]{BO}, which is
  compatible with Chern classes, these maps are identified after tensoring with
  $\Q$ with the two maps obtained by base change from
  $$
  \Phi _{E_0}^{\text{crys}}, \Phi _{F_0}^{\text{crys}}: \widetilde
  H_{\text{crys}}(X_0/W)\rightarrow \widetilde H_{\text{crys}}(Y_0/W).
  $$
  The result follows.
\end{proof}

\begin{pg} In the case when $X_n$ and $Y_n$ are K3 surfaces the action of
  $HH^*(X_n/R_n)$ on $HH_*(X_n/R_n)$ is faithful. Therefore from Lemma \ref{L:4.13} we
  obtain the following.
\end{pg}

\begin{cor}\label{C:4.15} Assume that $X_n$ and $Y_n$ are K3 surfaces and that
  $E_n, F_n\in D(X_n\times _{R_n}Y_n)$ are two $R_n$-perfect complexes satisfying
  condition $(\star )$. Suppose further that $\Phi ^{\text{\rm crys}}_{E_0}$ and
  $\Phi ^{\text{\rm crys}}_{F_0}$ are equal on the crystalline realizations of the
  Mukai motives of the reductions. Then $\Phi _{E_n}^{HH^*}$ and $\Phi
  _{F_n}^{HH^*}$ are equal.
\end{cor}
\begin{proof} Indeed since homology is a faithful module over cohomology the
  maps $\Phi _{E_n}^{HH^*}$ and $\Phi _{F_n}^{HH^*}$ are determined by the maps on
  Hochschild homology which are equal by Lemma \ref{L:4.13}.
\end{proof}

\begin{cor}\label{P:4.15} Let $X_{n+1}$ and $Y_{n+1}$ be K3 surfaces over
  $R_{n+1}$ and assume given an admissible  $R_{n+1}$-perfect complex $E_{n+1}$ on
  $X_{n+1}\times _{R_{n+1}}Y_{n+1}$ such that $E_n$ satisfies condition $(\star
  )$. Assume given an isomorphism $\sigma _n:X_n\rightarrow Y_n$ over $R_n$ such
  that the induced map $\sigma _0:X_0\rightarrow Y_0$ defines the same map on
  crystalline realizations of the Mukai motive as $E_0$. Then $\sigma _n$ lifts to
  an isomorphism $\sigma _{n+1}:X_{n+1}\rightarrow Y_{n+1}$.
\end{cor}
\begin{proof} Indeed by \eqref{E:keyformula} and the fact that $\Phi
  ^{HH^*}_{E_n}$ and $\Phi ^{HH^*}_{\Gamma _{\sigma _n}}$ are equal by
  Corollary \ref{C:4.15}, we see that the obstruction to lifting $\sigma _{n}$ is equal to
  the obstruction to lifting $E_n$, which is zero by assumption.
\end{proof}

\section{A remark on reduction types}\label{S:section4}

\begin{pg} In the proof of Theorem \ref{T:1.2} we need the following Theorem
  \ref{T:2.2v1}, whose proof relies on known characteristic $0$ results obtained
  from Hodge theory. In Section \ref{S:section10} below we give a different
  algebraic argument for Theorem \ref{T:2.2v1} in a special case which suffices
  for the proof of Theorem \ref{T:1.2}.
\end{pg}

\begin{pg} Let $V$ be a complete discrete valuation ring with field of fractions
  $K$ and residue field $k$. Let $X/V$ be a projective K3 surface with generic fiber $X_K$,
  and let $Y_K$ be a second K3 surface over $K$ such that the geometric fibers
  $X_{\overline K}$ and $Y_{\overline K}$ are derived equivalent.
\end{pg}

\begin{thm}\label{T:2.2v1} Under these assumptions the K3 surface $Y_K$ has
  potentially good reduction.
\end{thm}

\begin{rem} Here potentially good reduction means that after possibly replacing
  $V$ be a finite extension there exists a K3 surface $Y/V$ whose generic fiber
  is $Y_K$.
\end{rem}

\begin{proof}[Proof of Theorem \ref{T:2.2v1}] We use \cite[1.1 (1)]{LO} which implies
  that after replacing $V$ by a finite extension $Y_K$ is isomorphic to a moduli
  space of sheaves on $X_K$.

  After replacing $V$ by a finite extension we may assume that we have a complex
  $P\in D(X\times Y)$ defining an equivalence $$ \Phi _P:D(X_{\overline
    K})\rightarrow D(Y_{\overline K}). $$ Let $E\in D(Y\times X)$ be the complex
  defining the inverse equivalence $$ \Phi _E:D(Y_{\overline K})\rightarrow
  D(X_{\overline K}) $$ to $\Phi _P$. Let $\nu := \Phi _E(0,0,1)\in
  A^*(X_{\overline K})_{\text{num}, \Q}$ be the Mukai vector of a fiber of $E$ at
  a closed point $y\in Y_{\overline K}$ and write $$ \nu = (r, [L_X], s)\in
  A^0(X_{\overline K})_{\text{num}, \Q}\oplus A^1(X_{\overline K})_{\text{num},
    \Q}\oplus A^2(X_{\overline K})_{\text{num}, \Q}. $$ By \cite[8.1]{LO} we may
  after possible changing our choice of $P$, which may involve another extension
  of $V$, assume that $r$ is prime to $p$ and that $L_X$ is very ample. Making
  another extension of $V$ we may assume that $\nu $ is defined over $K$, and
  therefore by specialization also defines an element, which we denote by the same
  letter, $$ \nu = (r, [L_X], s)\in \Z\oplus \text{Pic}(X)\oplus \Z. $$ This class
  has the property that $r$ is prime to $p$ and that there exists another class
  $\nu '$ such that $\langle \nu, \nu '\rangle = 1$. This implies in particular
  that $\nu $ restricts to a primitive class on the closed fiber. Fix an ample
  class $h$ on $X$, and let $\mls M_h(\nu )$ denote the moduli space of semistable
  sheaves on $X$ with Mukai vector $\nu $. By \cite[3.16]{LO} the stack $\mls
  M_h(\nu )$ is a $\mu _r$-gerbe over a relative K3 surface $M_h(\nu )/V$, and
  by \cite[8.2]{LO} we have $Y_{\overline K}\simeq M_h(\nu )_{\overline K}$. In
  particular, $Y$ has potentially good reduction.
\end{proof}

\begin{rem} As discussed in \cite[p.~2]{LM}, to obtain Theorem \ref{T:2.2v1} it suffices
  to know that every K3 surface $Z_K$ over $K$ has potentially semistable
  reduction and this would follow from standard conjectures on resolution of
  singularities and toroidization of morphisms in mixed and positive
  characteristic. In the setting of Theorem \ref{T:2.2v1}, once we know that $Y_K$ has
  potentially semistable reduction then by \cite[Theorem on bottom of p. 2]{LM} we
  obtain that $Y_K$ has good reduction since the Galois representation
  $H^2(Y_{\overline K}, \Q_\ell )$ is unramified being isomorphic to direct
  summand of the $\ell $-adic realization $\widetilde H(X_{\overline K}, \Q_\ell
  )$ of the Mukai motive of $X_K$.
\end{rem}

\begin{pg}\label{P:4.6} One can also consider the problem of extending $Y_K$
  over a higher dimensional base. Let $B$ denote a normal finite type $k$-scheme
  with a point $s\in B(k)$ and let $X/B$ be a projective family of K3 surfaces.
  Let $K$ be the function field of $B$ and let $Y_K$ be a second K3 surface over
  $K$ Fourier-Mukai equivalent to $X_K$. Dominating $\mls O_{B, s}$ by a suitable
  complete discrete valuation ring $V$ we can find a morphism $$ \rho :\Sp
  (V)\rightarrow B $$ sending the closed point of $\Sp (V)$ to $s$ and an
  extension $Y_V$ of $\rho ^*Y_K$ to a smooth projective K3 surface over $V$. In
  particular, after replacing $B$ by its normalization in a finite extension of
  $K$ we can find a primitive polarization $\lambda _K$ on $Y_K$ of degree prime
  to the characteristic such that $\rho ^*\lambda _K$ extends to a polarization on
  $Y_V$. We then have a commutative diagram of solid arrows $$ \xymatrix{ \Sp
    (\text{Frac}(V))\ar@{^{(}->}[d]\ar[r]& \Sp (K)\ar@{^{(}->}[d]\ar@/^1pc/[rdd]& \\
    \Sp (V)\ar@/_1pc/[rrd]\ar[r]& B\ar@{-->}[rd]&\\ && \mls M_d} $$ for a suitable
  integer $d$. Base changing to a suitable \'etale neighborhood $U\rightarrow \mls
  M_d$ of the image of the closed point of $\Sp (V)$, with $U$ an affine scheme,
  we can after shrinking and possibly replacing $B$ by an alteration find a
  commutative diagram $$ \xymatrix{ \Sp (\text{Frac}(V))\ar@{^{(}->}[d]\ar[r]& \Sp
    (K)\ar@{^{(}->}[d]\ar@/^1pc/[rd]& \\ \Sp (V)\ar@/_2pc/[rr]\ar[r]& B\ar[rd]&
    U\ar@{^{(}->}[d]^-j\\ && \overline U,} $$ where $j$ is a dense open imbedding
  and $\overline U$ is projective over $k$. It follows that the image of $s$ in
  $\overline U$ in fact lands in $U$ which gives an extension of $Y_K$ to a
  neighborhood of $s$. This discussion implies the following:
\end{pg}

\begin{cor}\label{C:4.7} In the setup of Proposition \ref{P:4.6}, we can, after replacing
  $(B, s)$ by a neighborhood of a point in the preimage of $s$ in an alteration of
  $B$, find an extension of $Y_K$ to a K3 surface over $B$.
\end{cor}

\section{Supersingular reduction}\label{sec:supers-reduct}

\begin{pg}\label{P:5.1} Let $B$ be a normal scheme of finite type over an
  algebraically closed field $k$ of odd positive characteristic $p$. Let $K$
  denote the function field of $B$ and let $s\in B$ be a closed point. Let
  $f:X\rightarrow B$ be a projective K3 surface over $B$ and let $Y_K/K$ be a second K3
  surface over $K$ such that there exists a strongly filtered Fourier-Mukai
  equivalence $$ \Phi _P:D(X_K)\rightarrow D(Y_K) $$ defined by an object $P\in
  D(X_K\times _KY_K)$. Assume further that the fiber $X_s$ of $X$ over $s$ is a
  supersingular K3 surface.
\end{pg}

\begin{pg} Using Corollary \ref{C:4.7} we can, after possibly replacing $B$ by a
  neighborhood of a point over $s$ in an alteration, assume that we have a smooth
  K3 surface $Y/B$ extending $Y_K$ and an extension of the complex $P$ to a
  $B$-perfect complex $\mls P$ on $X\times _BY$, and furthermore that the complex
  $Q$ defining the inverse of $\Phi _P$ also extends to a complex $\mls Q$ on
  $X\times _BY$. Let $f_{\mls X}:\mls X\rightarrow B$ (resp. $f_{\mls Y}:\mls
  Y\rightarrow B$) be the structure morphism, and let $\mls
  H^i_{\text{crys}}(X/B)$ (resp. $\mls H^i_{\text{crys}}(Y/B)$) denote the
  $F$-crystal $R^if_{\mls X*}\mls O_{\mls X/W}$ (resp. $R^if_{\mls Y*}\mls O_{\mls
    Y/W}$) on $B/W$ obtained by forming the $i$-th higher direct image of the
  structure sheaf on the crystalline site of $\mls X/W$ (resp. $\mls Y/W$).
  Because $\Phi _{\mls P}$ is strongly filtered, it induces an isomorphism of
  $F$-crystals $$ \Phi ^{\text{crys}, i}_{\mls P}:\mls
  H^i_{\text{crys}}(X/B)\rightarrow \mls H^i_{\text{crys}}(Y/B) $$ for all $i$,
  with inverse defined by $\Phi _{\mls Q}$. Note that since we are working here
  with K3 surfaces these morphisms are defined integrally.

  We also have the de Rham realizations $\mls H^i_{\dR}(X/B)$ and $\mls H^i_\dR
  (Y/B)$ which are filtered modules with integrable connection on $B$ equipped
  with filtered isomorphisms compatible with the connections
  \begin{equation}\label{E:A} \Phi ^{\dR , i}_{\mls P}:\mls
    H^i_{\dR}(X/B)\rightarrow \mls H^i_{\dR}(Y/B). \end{equation} as well as \'etale
  realizations $\mls H^i_\et (X/B)$ and $\mls H^i_\et (Y/B)$ equipped with
  isomorphisms \begin{equation}\label{E:B} \Phi ^{\et , i}_{\mls P}:\mls H^i_\et
    (X/B)\rightarrow \mls H^i_\et (Y/B). \end{equation}
\end{pg}

\begin{pg} Let $H^i_{\text{crys}}(X_s/W)$ (resp. $H^i_{\text{crys}}(Y_s/W)$)
  denote the crystalline cohomology of the fibers over $s$. The isomorphism $\Phi
  ^{\text{crys}, i}_{\mls P}$ induces an isomorphism $$ \theta ^i
  :H^i_{\text{crys}}(X_s/W)\rightarrow H^i_{\text{crys}}(Y_s/W) $$ of
  $F$-crystals. By \cite[Theorem I]{Ogus2} this implies that $X_s$ and $Y_s$ are
  isomorphic. However, we may not necessarily have an isomorphism which induces
  $\theta ^2$ on cohomology.
\end{pg}

\begin{pg} Recall that as discussed in \cite[10.9 (iii)]{Huybrechts} if
  $C\subset X_K$ is a $(-2)$-curve then we can perform a spherical twist $$
  T_{\mls O_{C}}:D(X_K)\rightarrow D(X_K) $$ whose action on $NS(X_K)$ is the
  reflection $$ r_C(a) := a+\langle a, C\rangle C. $$
\end{pg}

\begin{prop}\label{T:5.5} After possibly changing our choice of model $Y$ for
  $Y_K$, replacing $(B, s)$ by a neighborhood of a point in an alteration over
  $s$, and composing with a sequence of spherical twists $T_{\mls O_C}$ along
  $(-2)$-curves in the generic fiber $Y_K$, there exists an isomorphism $\sigma
  :X_s\rightarrow Y_s$ inducing the isomorphism $\theta ^2$ on the second
  crystalline cohomology group. If $\theta ^2$ preserves the ample cone of the
  generic fiber then we can find an isomorphism $\sigma $ inducing $\theta ^2$.
\end{prop}

\begin{proof} By \cite[4.4 and 4.5]{Ogus3} there exists an isomorphism $\theta
  _0:NS(X_s)\rightarrow NS(Y_s)$ compatible with $\theta ^2$ in the sense that the
  diagram
  \begin{equation}\label{E:5.4.1} \xymatrix{ NS(X_s)\ar[r]^-{\theta
        _0}\ar[d]^-{c_1}& NS(Y_s)\ar[d]^-{c_1}\\ H^2_{\text{crys}}(X_s/W)\ar[r]^-{\theta
        ^2}& H^2_{\text{crys}}(Y_s/W)}
  \end{equation} commutes. Note that as discussed in \cite[4.8]{Liedtke} the map
  $\theta _0$ determines $\theta ^2$ by the Tate conjecture for K3 surfaces,
  proven by Charles, Maulik, and Pera \cite{Ch, Maulik, Pera}. In particular, if
  we have an isomorphism $\sigma :X_s\rightarrow Y_s$ inducing $\pm \theta _0$ on
  N\'eron-Severi groups then $\sigma $ also induces $\pm \theta ^2$ on crystalline
  cohomology. We therefore have to study the problem of finding an isomorphism
  $\sigma $ compatible with $\theta _0$.

  Ogus shows in \cite[Theorem II]{Ogus2} that there exists such an isomorphism
  $\sigma $ if and only if the map $\theta _0$ takes the ample cone to the ample
  cone. So our problem is to choose a model of $Y$ in such a way that $\pm \theta
  _0$ preserves ample cones. Set $$ V_{X_s}:= \{x\in NS(X_s)_{\R}|\text{$x^2>0$
    and $\langle x, \delta \rangle \neq 0$ for all $\delta \in NS(X_s)$ with $\delta
    ^2 = -2$}\}, $$ and define $V_{Y_s}$ similarly. Being an isometry the map
  $\theta _0$ then induces an isomorphism $V_{X_s}\rightarrow V_{Y_s}$, which we
  again denote by $\theta _0$. Let $R_{Y_s}$ denote the group of automorphisms of
  $V_{Y_s}$ generated by reflections in $(-2)$-curves and multiplication by $-1$.
  By \cite[Proposition 1.10 and Remark 1.10.9]{Ogus2} the ample cone of $Y_s$ is a
  connected component of $V_{Y_s}$ and the group $R_{Y_s}$ acts simply
  transitively on the set of connected components of $V_{Y_s}$.

  Let us show how to change model to account for reflections by $(-2)$-curves in
  $Y_s$. We show that after replacing $(P, Y)$ by a new pair $(P', Y')$ consisting
  of the complex $P\in D(X_K\times _KY_K)$ obtained by composing $\Phi _P$ with a
  sequence of spherical twists along $(-2)$-curves in $Y_K$ and replacing $Y$ by a
  new model $Y'$ there exists an isomorphism $\gamma :Y'_s\rightarrow Y_s$ such
  that the composition $$ \xymatrix{ NS(X_s)\ar[r]^-{\theta _0}&
    NS(Y_s)\ar[r]^-{r_C}& NS(Y_s)\ar[r]^-{\gamma ^*}& NS(Y'_s)} $$ is equal to the
  map $\theta '_0$ defined as for $\theta _0$ but using the model $Y'$.

  Let $C\subset Y_s$ be a $(-2)$-curve, and let
  $$ r_C:NS(Y_s)\rightarrow NS(Y_s)$$ be the reflection in the $(-2)$-curve.  
  If $C$ lifts to a curve in the family $Y$ we get a $(-2)$-curve in the generic
  fiber and so by replacing our $P$ by the complex $P'$ obtained by composition
  with the spherical twist by this curve in $Y_K$ (see \cite[10.9
  (iii)]{Huybrechts}) and setting $Y' = Y$ we get the desired new pair. If $C$
  does not lift to $Y$, then we take $P' = P$ but now replace $Y$ by the flop of
  $Y$ along $C$ as explained in \cite[2.8]{Ogus2}.

  Thus after making a sequence of replacements $(P, Y)\mapsto (P', Y')$ we can
  arrange that $\theta _0$ sends the ample cone of $X_s$ to plus or minus the
  ample cone of $Y_s$, and therefore we get our desired isomorphism $\sigma $.

  To see the last statement, note that we have modified the generic fiber by
  composing with reflections along $(-2)$-curves.  Therefore if $\lambda $ is an
  ample class on $X$ with restriction $\lambda _K$ to $X_K$, and for a general
  ample divisor $H$ we have $\langle \Phi _P(\lambda ), H\rangle >0$, then the
  same holds on the closed fiber. This implies that the ample
  cone of $X_s$ gets sent to the ample cone of $Y_s$ and not its negative.
\end{proof}

\begin{rem} One can also consider \'etale or de Rham cohomology in Theorem \ref{T:5.5}.
  Assume we have applied suitable spherical twists and chosen a model $Y$ such
  that we have an isomorphism $\sigma :X_s\rightarrow Y_s$ inducing $\pm \theta
  _0$. We claim that the maps $$ \theta _\dR :H_\dR ^i(X_s/k)\rightarrow H_\dR
  ^i(Y_s/k), \ \ \theta _\et :H_\et ^i(X_s, \Q_\ell )\rightarrow H_\et ^i(Y_s,
  \Q_\ell ) $$ induced by the maps \eqref{E:A} and \eqref{E:B} also agree with the
  maps defined by $\pm \sigma $. For de Rham cohomology this is clear using the
  comparison with crystalline cohomology, and for the \'etale cohomology it
  follows from compatibility with the cycle class map.
\end{rem}

\begin{pg} With notation as assumptions as in Theorem \ref{T:5.5} assume further
  that $B$ is a curve or a complete discrete valuation ring, and that we have
  chosen a model $Y$ such that each of the reductions satisfies condition $(\star
  )$ and such that the map $\theta _0$ in \eqref{E:5.4.1} preserves plus or minus
  the ample cones. Let $\sigma :X_s\rightarrow Y_s$ be an isomorphism inducing
  $\pm \theta _0$.
\end{pg}
\begin{lem}\label{P:5.7} The isomorphism $\sigma $ lifts to an isomorphism $\tilde
  \sigma :X\rightarrow Y$ over the completion $\widehat B$ at $s$ inducing the
  maps defined by $\pm \Phi ^{\text{\rm crys}, i} _{\mls P}$.
\end{lem}

\begin{proof} By Proposition \ref{P:2.10} in fact $\Phi _{\mls P}$ preserves the ample cone
  of the closed fiber and so we can choose $\sigma $ such that the map on
  cohomology is $\theta _0$. By Corollary \ref{P:4.15} $\sigma $ lifts uniquely to each
  infinitesimal neighborhood of $s$ in $B$, and therefore by the Grothendieck
  existence theorem we get a lifting $\tilde \sigma $ over $\widehat B$. That the
  realization of $\tilde \sigma $ on cohomology agrees with $\pm \Phi ^{\text{\rm
      crys}, i}_{\mls P}$ can be verified on the closed fiber where it holds by
  assumption.
\end{proof}

\begin{lem}\label{C:5.8} With notation and assumptions as in Proposition \ref{P:5.7} the map
  $\Phi _P^{A^*_{\text{\rm num}, \Q}}$ preserves the ample cones of the generic
  fibers.
\end{lem}
\begin{proof} The statement can be verified after making a field extension of
  the function field of $B$. The result therefore follows from Proposition \ref{P:5.7} and
  Proposition \ref{P:2.10}.
\end{proof}

\begin{rem}\label{R:6.10} In the case when the original $\Phi _P$ preserves the
  ample cones of the geometric generic fibers, no reflections along $(-2)$--curves
  in the generic fiber are needed. Indeed, by the above argument we get an
  isomorphism $\sigma _K:X_K\rightarrow Y_K$ such that the induced map on
  crystalline and \'etale cohomology
  agrees with $\Phi _P\circ \alpha $ for some sequence $\alpha $ of spherical
  twists along $(-2)$-curves in $X_K$ (also using Corollary \ref{C:5.8}). Since both $\sigma $
  and $\Phi _P$ preserve ample cones it follows that $\alpha $ also preserves the
  ample cone of $X_{\overline K}$. By \cite[1.10]{Ogus2} it follows that $\alpha $ acts
  trivially on the N\'eron-Severi group of $X_{\overline K}$. We claim that this implies that
  $\alpha $ also acts trivially on any of the cohomological realizations. We give
  the proof in the case of \'etale cohomology $H^2(X_{\overline K}, \Q_\ell )$ (for a prime
  $\ell $ invertible in $k$) leaving slight modifications for the other cohomology
  theories to the reader. Let $\widetilde R_X$ denote the subgroup of 
  $GL(H^2(X_{\overline K},\Q_\ell ))$ generated by $-1$ and the action induced by spherical twists along
  $(-2)$-curves in $X_{\overline K}$, and consider the inclusion of $\Q_\ell $-vector spaces
  with inner products $$ NS(X_{\overline K})_{\Q_\ell }\hookrightarrow H^2(X_{\overline K}, \Q_\ell ). $$ By
  \cite[Lemma 8.12]{Huybrechts} the action of the spherical twists along
  $(-2)$-curves in $X_{\overline K}$ on $H^2(X_{\overline K}, \Q_\ell )$ is by reflection across classes in
  the image of $NS(X_{\overline K})_{\Q_\ell }$. From this (and Gram-Schmidt!) it follows that
  the the group $\widetilde R_X$ preserves $NS(X_{\overline K})_{\Q_\ell }$, acts trivially on
  the quotient of $H^2(X_{\overline K}, \Q_\ell )$ by $NS(X_{\overline K})_{\Q_\ell }$, and that the
  restriction map $$ \widetilde R_X\rightarrow GL(NS(X_{\overline K})_{\Q_\ell }) $$ is
  injective. In particular, if an element $\alpha \in \widetilde R_X$ acts
  trivially on $NS(X_{\overline K})$ then it also acts trivially on \'etale cohomology.
  It follows that $\sigma $ and $\Phi _{P}$ induce the same map on realizations.
\end{rem}

\section{Specialization}\label{S:section7}

\begin{pg} We consider again the setup of Proposition \ref{P:5.1}, but now we don't assume
  that the closed fiber $X_s$ is supersingular. Further we restrict attention to
  the case when $B$ is a smooth curve, and assume we are given a smooth model
  $Y/B$ of $Y_K$ and a $B$-perfect complex $\mls P\in D(X\times _BY)$ such that
  for all geometric points $\bar z \rightarrow B$ the induced complex $\mls P_{\bar
    z}$ on $X_{\bar z}\times Y_{\bar z}$ defines a strongly filtered equivalence
  $D(X_{\bar z})\rightarrow D(Y_{\bar z})$.

  Let $\mls H^i(X/B)$ (resp. $\mls H^i(Y/B)$) denote either $\mls H^i_\et (X/B)$
  (resp. $\mls H^i_\et (Y/B)$) for some prime $\ell \neq p$ or $\mls H^i_\text{\rm
    crys} (X/B)$ (resp. $\mls H^i_\text{\rm crys} (Y/B)$). Assume further given an
  isomorphism $$ \sigma _K:X_K\rightarrow Y_K $$ inducing the map given by
  restricting $$ \Phi ^{ i}_{\mls P}:\mls H^i(X/B)\rightarrow \mls H^i(Y/B) $$ to
  the generic point.
\end{pg}

\begin{rem} If we work with \'etale cohomology in this setup we could also
  consider the spectrum of a complete discrete valuation ring instead of $B$, and
  in particular also a mixed characteristic discrete valuation ring.
\end{rem}

\begin{rem} When the characteristic of $k$ is zero we can also use de Rham
  cohomology instead of \'etale cohomology.
\end{rem}

\begin{prop}\label{P:6.4} The isomorphism $\sigma _K $ extends to an isomorphism
  $\sigma :X\rightarrow Y$.
\end{prop}
\begin{proof} We give the argument here for \'etale cohomology in the case when
  $B$ is the spectrum of a discrete valuation ring, leaving the minor
  modifications for the other cases to the reader.

  Let $Z\subset X\times _BY$ be the closure of the graph of $\sigma _K$, so $Z$
  is an irreducible flat $V$-scheme of dimension $3$ and we have a correspondence
  $$ \xymatrix{ & Z\ar[ld]_-{p}\ar[rd]^-q& \\ X&& Y.} $$ Fix an ample line bundle
  $L$ on $X$ and consider the line bundle $M:= \text{det}(Rq_*p^*L)$ on $Y$. The
  restriction of $M$ to $Y_K$ is simply $\sigma _{K*}L$, and in particular the
  \'etale cohomology class of $M$ is equal to the class of $ \Phi _{\mls P}(L)$.
  By our assumption that $\Phi _{\mls P}$ is strongly filtered in the fibers the
  line bundle $M$ is ample on $Y$. Note also that by our assumption that $\Phi
  _{\mls P}$ is strongly filtered in every fiber we have $$ \Phi _{\mls
    P}(L^{\otimes n})\simeq \Phi ^{\mls P}(L)^{\otimes n}. $$ In particular we can
  choose $L$ very ample in such a way that $M$ is also very ample.  The result then follows from Matsusaka-Mumford \cite[Theorem 2]{MM}.
\end{proof}

\section{Proof of Theorem \ref{T:1.2}}\label{S:section8}

\begin{pg} Let $K$ be an algebraically closed field extension of $k$ and let $X$
  and $Y$ be K3 surfaces over $K$ equipped with a complex $P\in D(X\times _KY)$
  defining a strongly filtered Fourier-Mukai equivalence $$ \Phi
  _P:D(X)\rightarrow D(Y). $$ We can then choose a primitive polarization $\lambda
  $ on $X$ of degree prime to $p$ such that the triple $((X, \lambda ), Y, P)$
  defines a $K$-point of $\mls S_d$. In this way the proof of Theorem \ref{T:1.2} is
  reformulated into showing the following: For any algebraically closed field $K$
  and point $((X, \lambda ), Y, P)\in \mls S_d(K)$ there exists an isomorphism
  $\sigma :X\rightarrow Y$ such that the maps on crystalline and \'etale
  realizations defined by $\sigma $ and $\Phi _P$ agree.
\end{pg}

\begin{pg}\label{P:extendfield}
  To prove this it suffices to show that there exists such an isomorphism after replacing $K$ by a field extension.  To see this  let $I$ denote the scheme of isomorphisms between $X$ and $Y$, which is a
  locally closed subscheme of the Hilbert scheme of $X\times _KY$. Over $I$ we
  have a tautological isomorphism $\sigma ^u:X_I\rightarrow Y_I$. The condition
  that the induced action on $\ell $-adic \'etale cohomology agrees with $\Phi _P$
  is an open and closed condition on $I$. It follows that there exists a subscheme
  $I'\subset I$ classifying isomorphisms $\sigma $ as in the theorem. This implies
  that if we can find an isomorphism $\sigma $ over a field extension of $K$ then
  such an isomorphism also exists over $K$. 
\end{pg}

\begin{pg} By Proposition \ref{P:6.4} it suffices to show that the result holds for each
  generic point of $\mls S_d$. By Theorem \ref{T:4.8} any such generic point maps to a
  generic point of $\mls M_d$ which by Theorem \ref{T:2.2} admits a specialization to a
  supersingular point $x\in \mls M_d(k)$ given by a family $(X_R, \lambda _R)/R$,
  where $R$ is a complete discrete valuation ring over $k$ with residue field
  $\Omega $, for some algebraically closed field $\Omega $.  By Theorem \ref{T:2.2v1} the point $(Y, \lambda _Y)\in \mls M_d(K)$ also has a limit $y\in \mls M_d(\Omega )$ given by a second family $(Y_R, \lambda _R)/R$.  Let $P'$ be the complex on $X\times Y$ giving the composition of $\Phi _P$ with suitable twists by $(-2)$-curves such that after replacing $Y_R$ by a sequence of flops the map $\Phi _{P'}$ induces an  isomorphism on crystalline cohomology on the closed fiber preserving plus or minus the ample cone.  By the Cohen structure theorem we have $R\simeq \Omega [[t]]$, and $((X, \lambda ), Y, P')$ defines a point of $\mls S_d(\Omega ((t)))$.

  Let $B$ denote the completion of the strict henselization of $\mls M_{\Z[1/d]}\times \mls M_{\Z[1/d]}$ at the point $(x, y)$.  So $B$ is a regular complete local ring with residue field $\Omega $.  Let $B'$ denote the formal completion of the strict henselization of $\overline {\mls S}_{d, \Z[1/d]}$ at the $\Omega ((t))$-point given by $((X, \lambda ), Y, P')$.  So we obtain a commutative diagram
  \begin{equation}\label{E:ringmagic}
    \xymatrix{
      B\ar[r]\ar[d]& \Omega [[t]]\ar[d]\\
      B'\ar[r]& \Omega ((t)).}
  \end{equation}
  Over $B$ we have a universal families $\mls X_B$ and $\mls Y_B$, and over the base changes to $B'$ we have, after trivializing the pullback of the gerbe $\mls S_{d, \Z[1/d]}\rightarrow \overline {\mls S}_{d, \Z[1/d]}$, a complex $\mls P_{B'}'$ on $\mls X_{B'}\times _{B'}\mls Y_{B'}$, which reduces to the triple $(X, Y, P')$ over $\Omega ((t))$.   The map $B\rightarrow B'$ is a filtering direct limit of \'etale morphisms.  We can therefore replace $B'$ by a finite type \'etale $B$-subalgebra over which all the data is defined and we still have the diagram \eqref{E:ringmagic}.  Let $\overline B$ denote the integral closure of $B$ in $B'$ so we have a commutative diagram
  $$
  \xymatrix{
    \Sp (B')\ar@{^{(}-}[r]\ar[rd]& \Sp (\overline B)\ar[d]\\
    & \Sp (B),}
  $$
  where $\overline B$ is flat over $\Z[1/d]$ and normal.  Let $Y\rightarrow \Sp (\overline B)$ be an alteration with $Y$ regular and flat over $\Z[1/d]$, and let $Y'\subset Y$ be the preimage of $\Sp (B')$.    Lifting the map $B\rightarrow \Omega [[t]]$ to a map $\Sp (\widetilde R)\rightarrow Y$ for some finite extension of complete discrete valuation rings $\widetilde R/R$ and letting $C$ denote the completion of the local ring of $Y$ at the image of the closed point of $\Sp (\widetilde R)$ we obtain a commutative diagram
  $$
  \xymatrix{
    C\ar[r]\ar[d]& \Omega [[t]]\ar[d]\\
    C'\ar[r]& \Omega ((t)),}
  $$
  where $C\rightarrow C'$ is a localization, we have K3-surfaces $\mls X_C$ and $\mls Y_C$ over $C$ and a perfect complex $\mls P_{C'}'$ on $\mls X_{C'}\times _{C'}\mls Y_{C'}$ defining a Fourier-Mukai equivalence and the triple $(\mls X_{C'}, \mls Y_{C'}, \mls P_{C'}')$ reducing to $(X, Y, P)$ over $\Omega ((t))$.  By \cite[5.2.2]{TT} we can extend the complex $\mls P'_{C'}$ to a $C$-perfect complex $\mls P'_C$ on $\mls X_C\times _C\mls Y_C$ (here we use that $C$ is regular).  It follows that the base change $(X_{\Omega [[t]]}, Y_{\Omega [[t]]}, P'_{\Omega [[t]]})$ gives an extension of $(X, Y, P)$ to $\Omega [[t]]$ all of whose reductions satisfy  our condition $(\star )$.  

  This puts us in the setting of Proposition \ref{P:5.7}, and we conclude that there exists an isomorphism
  $\sigma
  :X\rightarrow Y$ (over $\Omega ((t))$, but as noted above we are allowed to make a field extension of $K$) such that the induced map on crystalline and \'etale cohomology
  agrees with $\Phi _P\circ \alpha $ for some sequence $\alpha $ of spherical
  twists along $(-2)$-curves in $X$ (using also Corollary \ref{C:5.8}). By the
  same argument as in Remark \ref{R:6.10}  it follows that $\sigma $ and $\Phi _{P}$ induce the same map on realizations
  which concludes the proof of Theorem \ref{T:1.2}. \qed
\end{pg}

\begin{rem} One consequence of the proof is that in fact any strongly filtered
  equivalence automotically takes the ample cone to the ample cone, and not its
  negative. This is closely related to \cite[4.1]{HMS}.
\end{rem}

\section{Characteristic $0$}\label{S:section9}

From our discussion of positive characteristic results one can also deduce the
following result in characteristic $0$.

\begin{thm}\label{T:8.1} Let $K$ be an algebraically closed field of
  characteristic $0$, let $X$ and $Y$ be K3 surfaces over $K$, and let $\Phi
  _P:D(X)\rightarrow D(Y)$ be a strongly filtered Fourier-Mukai equivalence
  defined by an object $P\in D(X\times Y)$. Then there exists an isomorphism
  $\sigma :X\rightarrow Y$ whose action on $\ell $-adic and de Rham cohomology
  agrees with the action of $\Phi _P$.
\end{thm}

\begin{proof} It suffices to show that we can find an isomorphism $\sigma $
  which induces the same map on $\ell $-adic cohomology as $\Phi _P$ for a single
  prime $\ell $. For then by compatibility of the comparison isomorphisms with
  $\Phi _P$, discussed in \cite[\S 2]{LO}, it follows that $\sigma $ and $\Phi _P$
  also define the same action on the other realizations of the Mukai motive.

  Furthermore as in Proposition \ref{P:extendfield} it suffices to prove the existence of $\sigma $ after making a field extension of $K$.

  As in Proposition \ref{P:extendfield} let $I'$ denote the scheme of isomorphisms $\sigma :X\rightarrow Y$ as in the theorem.  Note that since the action of such $\sigma $ on the ample cone is fixed, the scheme $I'$ is in fact of finite type.

  Since $X$, $Y$, and $P$ are all locally finitely presented over $K$ we can
  find a finite type integral $\Z$-algebra $A$, K3 surfaces $X_A$ and $Y_A$ over
  $A$, and an $A$-perfect complex $P_A\in D(X_A\times _AY_A)$ defining a strongly
  filtered Fourier-Mukai equivalence in every fiber, and such that $(X, Y, P)$ is
  obtained from $(X_A, Y_A, P_A)$ by base change along a map $A\rightarrow K$.
  The scheme $I'$ then also extends to a finite type $A$-scheme $I'_A$ over $A$.
  Since $I'$ is of finite type over $A$ to prove that $I'$ is nonempty it
  suffices to show that $I'_A$ has nonempty fiber over $\mathbb{F}_p$ for
  infinitely many primes $p$.  This holds by Theorem \ref{T:1.2}.
\end{proof}


\section{Bypassing Hodge theory}\label{S:section10}

\begin{pg}
  The appeal to analytic techniques implicit in the results of Section
  \ref{S:section4}, where characteristic $0$ results based on Hodge theory are
  used to deduce Theorem \ref{T:2.2v1}, can be bypassed in the following way using
  results of \cite{Maulik} and \cite{LM}.
\end{pg}

\begin{pg} Let $R$ be a complete discrete valuation ring of equicharacteristic
  $p>0$ with residue field $k$ and fraction field $K$. Let $X/R$ be a smooth K3
  surface with supersingular closed fiber. Let $Y_K$ be a K3 surface over $K$
  and $P_K\in D(X_K\times Y_K)$ a perfect complex defining a Fourier-Mukai
  equivalence $\Phi _{P_K}:D(X_{\overline K})\rightarrow D(Y_{\overline K})$.
\end{pg}

\begin{thm}\label{T:GR} Assume that $X$ admits an ample invertible sheaf $L$
  such that $p>L^2+4$. Then after replacing $R$ by a finite extension there exists
  a smooth projective K3 surface $Y/R$ with generic fiber $Y_K$.
\end{thm}
\begin{proof}
  Changing our choice of Fourier-Mukai equivalence $P_K$, we may assume that
  $P_K$ is strongly filtered. Setting $M_K$ equal to $\text{det}(\Phi _{P_K}(L))$
  or its dual, depending on whether $\Phi _{P_K}$ preserves ample cones, we get an
  ample invertible sheaf on $Y_K$ of degree $L^2$. By \cite[2.2]{LM}, building on
  Maulik's work \cite[Discussion preceding 4.9]{Maulik} we get a smooth K3 surface
  $Y/R$ with $Y$ an algebraic space. Now after replacing $P_K$ by the composition
  with twists along $(-2)$-curves and the model $Y$ by a sequence of flops, we can
  arrange that the map on crystalline cohomology of the closed fibers induced by
  $\Phi _{P_K}$ preserves ample cones. Let $P\in D(X\times _RY)$ be an extension
  of $P_K$ and let $M$ denote $\text{det}(\Phi _P(L))$. Then $M$ is a line bundle
  on $Y$ whose reduction is ample on the closed fiber. It follows that $M$ is also
  ample on $Y$ so $Y$ is a projective scheme.
\end{proof}

\begin{pg} We use this to prove Theorem \ref{T:1.2} in the case of \'etale realization
  in the following way. First observe that using the same argument as in Section
  \ref{S:section8}, but now replacing the appeal to Theorem \ref{T:2.2v1} by the above
  Theorem \ref{T:GR}, we get Theorem \ref{T:1.2} under the additional assumption that $X$
  admits an ample invertible sheaf $L$ with $p>L^2+4$. By the argument of
  Section \ref{S:section9} this suffices to get Theorem \ref{T:1.2} in characteristic $0$,
  and by the specialization argument of Section \ref{S:section7} we then get also the
  result in arbitrary characteristic.
\end{pg}

\providecommand{\bysame}{\leavevmode\hbox to3em{\hrulefill}\thinspace}

\end{document}